\documentclass[11pt]{article}
\usepackage{geometry}                
\usepackage{amsmath, amsthm, amssymb}
\usepackage{graphicx, color}
\usepackage{pdfsync}
\newtheorem{lemma}{Lemma}

\newtheorem{proposition}[lemma]{Proposition}
\newtheorem{remark}[lemma]{Remark}

\newtheorem{theorem}[lemma]{Theorem}

\usepackage{graphicx}
\usepackage{amssymb}
\usepackage{epstopdf}
\newcommand{\e}{\varepsilon}
  
\newcommand{\disp}{\displaystyle}

\newcommand{\weakstar}{\buildrel * \over \rightharpoonup}

\def\ZZ{\mathbb Z}

\title{Asymptotic analysis of a ferromagnetic lattice spin system with diffuse interfacial energy}
\author{Andrea Braides\\ \small Dipartimento di Matematica, Universit\`a di Roma Tor Vergata
\\ \small  via della ricerca scientifica 1, 00133 Roma, Italy\\ \\  Andrea Causin and Margherita Solci \\ \ \small 
DADU, Universit\`a di Sassari\\ \small 
 piazza Duomo 6, 07041 Alghero (SS), Italy}\date{}                                           

\begin{document}
\maketitle

\section{Introduction}
Despite the vast Statistical Mechanics literature on spin systems, in particular those
parameterized on a lattice, their treatment from the variational standpoint is relatively 
recent and still incomplete. In the simplest case, such systems can be seen as driven by an energy
\begin{equation}\label{SMform}
-\sum_{i,j} c_{ij} u_iu_j,
\end{equation}
where $i,j$ belong to some subset of a lattice $\mathcal L$ and the variable $u_i$ takes the value 
in $\{-1,+1\}$. 
Note that in the ferromagnetic case (i.e., when $c_{ij}\ge 0$ for all $i,j$), up to additive and multiplicative constants, it is more handy to rewrite the energies in (\ref{SMform}) as
\begin{equation}\label{CVform}
\sum_{i,j} c_{ij} (u_i-u_j)^2
\end{equation}
in order to have minimizers with 
zero energy and to avoid $+\infty-\infty$ indeterminate forms in the case of infinite domains.
The paper \cite{CDL} by Caffarelli and de la Llave provided a first homogenization result for periodic ferromagnetic spin systems
by characterizing ground states as plane-like minimizers (see also the recent paper \cite{B}).

 A later work by Alicandro {\em et al.}~\cite{ABC} formalized the treatment of such systems in terms of $\Gamma$-convergence.
Those authors set the problem in the framework of a discrete-to-continuum analysis, scaling the energies and 
characterizing the continuum limit within the theory of interfacial energies defined on 
partitions by Ambrosio and Braides \cite{AB}, and also treating some antiferromagnetic case. 
In that approach the energies are scaled by a small parameter $\e>0$ as
\begin{equation}\label{CVform-e}
\sum_{i,j}  \e^{d-1}c_{ij}(u_i-u_j)^2,
\end{equation}
where now $i,j$ are supposed to belong to $\e\mathcal L$ and $d$ is the dimension of the ambient space; i.e.,  $\mathcal L\subset \mathbb R^d$.
Note the $d-1$-th power in (\ref{CVform-e}), corresponding to a {\em surface scaling}.

In the case of ferromagnetic systems, the macroscopic magnetization parameter $u$ still only takes the
values $\pm 1$, and the continuum-limit energy has the form
\begin{equation}\label{Contform}
\int_{\partial\{u=1\}}\varphi(\nu)d{\mathcal H}^{d-1},
\end{equation}
where $\nu$ is the measure-theoretical normal to the set of finite perimeter $\{u=1\}$.
A general homogenization theorem within the class of ferromagnetic spin system was
proved by Braides and Piatnitski also allowing for random coefficients \cite{BP}. 
In all those cases the limit problem is of the form (\ref{Contform}). Applications
of this result comprise the description of quasicrystalline structures \cite{BCS,BS}
and optimal design problems for networks \cite{BK,BK2}. 
Moreover, the homogenization result for periodic systems has been recently extended to some types of 
antiferromagnetic interactions when the limit is instead parameterized on partitions
into sets of finite perimeter \cite{BC} and can be written as a sum of energies
of the form (\ref{Contform}).

Alicandro and Gelli 
\cite{AG} have recently remarked that if we take $\e$-depending coefficients;
i.e., we allow for energies of the form
\begin{equation}\label{CVform-ee}
\sum_{i,j} c^\e_{ij}(u_i-u_j)^2
\end{equation}
 in place of (\ref{CVform}), still within the framework of the discrete-to continuum analysis of
ferromagnetic systems with a surface scale scaling, the limit functional
might contain non-local terms,
of the form
\begin{equation}\label{CVform-nonloc}
\int k(x,y, u(x)-u(y))d\mu(x,y).
\end{equation}

Different scalings of the energies are also possible. In \cite{ACG} Alicandro {\em et al.}~have examined the {\em bulk scaling} 
\begin{equation}\label{CVform-bulk}
\sum_{i,j} \e^{d}f(u_i,u_j),
\end{equation}
for general $f$ (and $u_i$ taking values in a more general set),
showing that the limit is a bulk integral 
\begin{equation}\label{CVform-bulklim}
\int g^{**}(u)\,dx
\end{equation}
($g^{**}$ denotes the convex envelope of $g$).
Note however that in the ferromagnetic spin case $g$ is trivial, and can be interpreted as
a double-well potential with minima in $\pm1$. As a consequence, when the hypothesis of \cite{AG} are satisfied, formally, 
ferromagnetic spin systems can be approximated in the continuum as an expansion
\begin{equation}\label{CVform-develop}
\int g(u)\,dx+\e\int_{\partial\{u=1\}}\varphi(\nu)d{\mathcal H}^{d-1}+\ldots
\end{equation}
(see \cite{GCB,BT}).
This expression highlights a {\em separation of scales} effect, which suggests that either
a bulk or a surface scaling have to be taken into account (depending on the problem at hand),
unless higher-order scalings come into play.

In this paper we show that this is not the case for general $\e$-depending spin energies with a 
simple example in dimension one. In the notation above, the energies we will 
examine can be written as in (\ref{CVform-ee})
with $c^\e_{ij}=1$ if $|i-j|=1$ or $|i-j|=\lfloor1/\sqrt\e\rfloor$, and $c^\e_{ij}=0$ otherwise. 
These energies have long-range interactions which do not satisfy the decay conditions 
on $c^\e_{ij}$ required in \cite{AG} to obtain an integral representation
as an interfacial energy. We will instead show that energies (\ref{CVform-ee}) have a
meaningful limit, which is not of the forms described above, at an intermediate scale between the bulk and surface scales
(namely, at the scale $\sqrt\e$).
More precisely, if we examine the discrete-to-continuum limit of
\begin{equation}\label{e1}
\sum_{i,j} \sqrt\e \, c^\e_{ij}(u_i-u_j)^2,
\end{equation}
restricted to the portion of $\e\ZZ$ contained in an interval $(a,b)$ 
then the continuum parameter $u\in BV((a,b);[-1,1])$ is a function with bounded variation
taking values in $[-1,1]$, and, denoted by $Du$ its derivative in the sense of distributions
(which is a measure on $(a,b)$), we have a limit energy
\begin{equation}\label{ce1}
2|\{x\in(a,b): -1<u(x)<1\}|+|Du|(a,b).
\end{equation}
Hence, the effect of the interaction coefficients is not strong enough to force that $u\in\{\pm1\}$
but strong enough to give a dependence on $|Du|$.

\begin{figure}[h!]
\centerline{\includegraphics [width=5in]{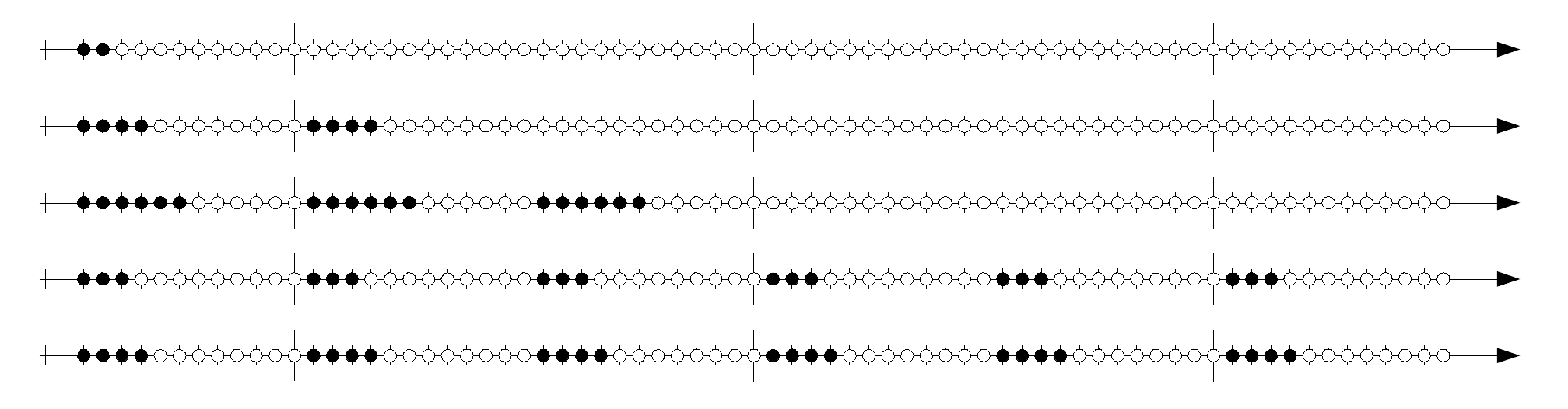}}
\caption{Optimal configurations for different volume fractions}\label{recsequence}
   \end{figure}
By examining the minimizers of the continuum limit we highlight interesting features of the optimal arrangements for the discrete energy. In Fig.~\ref{recsequence} we picture discrete minimizers with prescribed integral $\sigma$ of $u$ and for $b-a$ small. For small values of $\sigma$ we have 
a partial concentration on one side of the interval, which grows until a certain threshold, after which the function $u$ is (approximately) periodically distributed over the whole interval. We note that this description is much easier when obtained from the continuum energy.

\medskip
The idea behind the proof of the continuum approximation is that the energy in (\ref{CVform-ee}) can be equivalently interpreted 
as defined on the two-dimensional lattice $\lfloor1/\sqrt\e\rfloor\ZZ^2$
if nearest-neighbours in $\ZZ$ are interpreted as nearest-neighbours
in the vertical direction in $\lfloor1/\sqrt\e\rfloor\ZZ^2$
and, correspondingly, $\lfloor1/\sqrt\e\rfloor$-neighbours in $\ZZ$ as nearest-neighbours in the horizontal direction  in $\lfloor1/\sqrt\e\rfloor\ZZ^2$. With this identification
the energy becomes
a simple nearest-neighbour interaction energy in dimension two, 
of which we can compute the $\Gamma$-limit in the surface scaling.
Reinterpreting the limit in dimension one gives the form (\ref{ce1}) after some technical arguments.

The interest in this example is that the limit is characterized by the non-trivial topology of the graph of the connections $i,j$ with $c^\e_{ij}=1$, which is the same as that of nearest-neighbours in dimension two.
This is an argument different from the measure-theoretical ones used in the previous articles cited above.

We complement the analysis with a study of the minimum problems for the limit energy both when a volume constraint is taken into account and when periodic conditions are imposed, thus recovering the behaviour
of minimizers for the discrete problems by $\Gamma$-convergence. It is interesting to note the complex structure of the minimum energy landscape in dependence of the parameters of the problem, and in particular a size effect highlighted by the dependence on the width of the underlying interval. Furthermore, in the periodic case another parameter intervenes given by the ``defect'' of $\lfloor1/\sqrt\e\rfloor$-periodicity of the interval (normalized to a number between $0$ and $1$). Correspondingly, a boundary term must be added to the $\Gamma$-limit,
which further influences the shape of minimizers for certain values of that parameter.

\section{Statement of the result}
For the sake of notational simplicity, we consider a discrete parameter $n\in \mathbb N$ in the place of $\e$.  
In the notation used in the Introduction we choose $\e={1\over n^2}$, so that ${1\over\sqrt\e}=n$.
Moreover, we will consider spin functions with values in $\{0,1\}$ instead of $\{-1,1\}$. 

For each $n\in \mathbb N$ we define
\begin{equation}\label{def-en}
\mathcal E_n=\Big\{ \{i,j\}: \ i, j \in \mathbb N \cap (0, n^2] \hbox{ and } |i-j|= 1 \hbox{ or } |i-j|=n\Big\}.
\end{equation}
This set of indices corresponds to $i,j$ such that $c^\e_{ij}\neq 0$ in (\ref{e1}).

Let $\mathcal A_n$ be the set of the functions $u\colon(0,1]\to\{0,1\}$ with 
$u$ constant on each interval $(\frac{i-1}{n^2},\frac{i}{n^2}],$ $i=1,\ldots, n^2$.
Such a $u$ corresponds to a discrete function, which we still denote by $u$,
defined as its restriction to $
\big(\frac{1}{n^2}\mathbb N\big)\cap(0,1]$. 
We will denote $u_i$ for $u\big(\frac{i}{n^2}\big)$.  

Now, for $u\in \mathcal A_n$ we define the functional 
\begin{equation}\label{def-fn}
F_n(u)=\frac{1}{n}\sum_{\{i,j\}\in \mathcal E_n} (u_i-u_j)^2.
\end{equation}
This is a rewriting of energy (\ref{e1}), with the scaling ${1\over n}$, and with the constraint
$u_i\in\{0,1\}$ instead of $u_i\in\{-1,+1\}$.
Note that this corresponds to a scaling $\sqrt \e$, which is intermediate between
the bulk scaling $\e$ and the surface scaling (since the latter corresponds to no scaling
of the energy in dimension one).

We prove the following $\Gamma$-convergence result.  
\begin{theorem} 
The sequence of functionals $\{F_n\}$ $\Gamma$-converges with respect to the weak-$*$ convergence in $L^\infty(0,1)$ 
to the functional $F\colon L^\infty(0,1)\to [0,+\infty]$ given by 
\begin{equation}\label{def-f}
F(u)=\left\{\begin{array}{ll}
2\,\mathcal H^1(\{ x: 
0<u(x)<1 \})+|Du|(0,1) & \hbox{ if } u\in BV(0,1), \ 0\leq u\leq 1   \\ 
+\infty & \hbox{ otherwise.}
\end{array}
\right.
\end{equation} 
\end{theorem}

  \begin{figure}[h!]
\centerline{\includegraphics [width=6in]{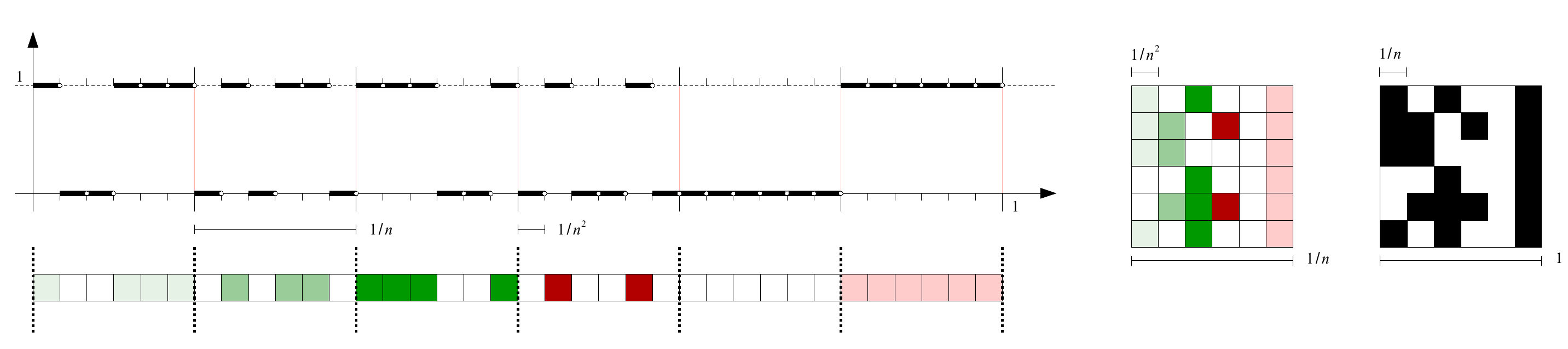}}
\caption{Identification of a spin function with a set in ${\mathbb R}^2$}\label{R2}
   \end{figure}
   To explain the form of the limit energies it is convenient to reinterpret the energies $F_n$ in a two-dimensional setting. Indeed, the decomposition 
$\frac{i}{n^2}=\frac{h_i-1}{n}+\frac{k_i}{n^2}$ with $h_i,k_i\in\{1,\dots, n\}$ induces a one-to-one correspondence between $\big(\frac{1}{n^2}\mathbb N\big)\cap(0,1]$ and 
$\big(\frac{1}{n}\mathbb N^2\big)\cap(0,1]^2$ given by $\frac{i}{n^2}\mapsto \frac{1}{n}(h_i,k_i)$. With this construction, we map each interval $(\frac{i-1}{n^2},\frac{i}{n^2}]$ to the square $(\frac{h_i-1}{n}, \frac{h_i}{n}]\times(\frac{k_i-1}{n}, \frac{k_i}{n}]$ 
so that any function $u\in\mathcal A_n$ can be represented as the characteristic function of a subset of $(0,1]^2$ (see Fig.~\ref{R2}). 
Hence, we will study the asymptotic behaviour of the sequence $F_n$ as $n\to+\infty$  
by using the results for the $\Gamma$-convergence for
nearest-neighbour interaction energies  
in the two-dimensional square lattice $\frac{1}{n}\mathbb Z^2$. 
These energies are defined  by 
\begin{equation}\label{def-en}
E_n(v; \Omega)=\frac{1}{n}\sum_{\{z,z^\prime\}\in \mathcal N(n\Omega)}(v_z-v_{z^\prime})^2
\end{equation} 
for $v\colon  \frac{1}{n}\mathbb Z^2\to \{0,1\}$  
and $\Omega$ a Lipschitz open subset of $\mathbb R^2$. 
The sum is running over the set $\mathcal N(n\Omega)$ 
of the pairs of nearest neighbours in $n\Omega\cap\mathbb Z^2$, 
and $v_z$ stands for $v(\frac{z}{n})$. The behaviour of $E_n$ is characterized 
by the following result \cite{CDL,ABC}, where it is understood that each function $v$ is extended 
as a piecewise-constant function to each square $z+[0,{1\over n})^2$. 
\begin{proposition}\label{nn2} 
The sequence $\{E_n\}$ is equicoercive on $L^1_{\rm loc} (\Omega)$. Its $\Gamma$-limit in the 
strong $L^1(\Omega)$ convergence is finite only on functions $u\in BV(\Omega;\{0,1\})$; i.e., 
on characteristic functions of sets of finite perimeter, and its value is
\begin{equation}
F^1(u)=\hbox{\rm Per}_1(\{u=1\}; \Omega):=\int_{\partial^*\{u=1\}} \|\nu_u\|_1d{\cal H}^1,
\end{equation}
where $\nu_u$ denotes the interior normal to $\partial^*\{u=1\}$ and $\|\nu\|_1=|\nu_1|+|\nu_2|$.
\end{proposition}

\section{Proof of the result}
We separately prove the upper and lower bounds for the $\Gamma$-limit.
\begin{proposition}[Lower bound]\label{proof-liminf}
Let $\{u_n\}$ be a sequence with $u_n\in \mathcal A_n$ such that $F_n(u_n)$ is equibounded. 
Then there exists $u\in BV(0,1)$ such that,  
up to subsequences, $u_n\weakstar u$ in $L^\infty(0,1)$
and 
\begin{equation}\label{liminf-eq}
\liminf_{n\to+\infty} F_n(u_n)\geq F(u). 
\end{equation}
\end{proposition}
\begin{proof}
The key point of the proof is the construction of a suitable sequence of functions 
$v_n\colon \frac{1}{n}\mathbb Z^2\to \{0,1\}$ 
such that for any $\varrho>0$ and for a suitable choice of $\eta$ we have 
$$F_n(u_n; (\varrho, 1-\varrho))\geq E_n(v_n; Q_\varrho^\eta)$$ 
provided that $n > \frac{1}{\varrho}$, where $Q_\varrho^\eta=(\varrho,1-\varrho)\times(-\eta,1-\eta)$ 
and $E_n$ is the nearest-neighbour interaction energy defined in (\ref{def-en}).

Let $\{u_n\}$ be such that $u_n\in \mathcal A_n$ and $F_n(u_n) \leq c<+\infty.$ 
Since $\{u_n\}$ is equibounded in $L^\infty(0,1)$ we can assume that (up to subsequences) 
$u_n\weakstar u$ in $L^\infty(0,1)$. Denoting  for all $j=1,\ldots, n$ by
$\alpha_n^j$ the integral mean $$\alpha_n^j=n \int_{I_n^j} u_n(t)\, dt={1\over n}\sum_{k=1}^n u_n\Bigl({j-1\over n}+{k\over n^2}\Bigr),$$
where $I_n^j=(\frac{j-1}{n}, \frac{j}{n}]$, we define $\hat{u}_n$ in each $I_n^j$ by setting 
\begin{equation}\label{def-hatun}
\hat u_n= \left\{ \begin{array}{ll}
1 & \hbox{ in } (\frac{j-1}{n}, \frac{j-1}{n}+\frac{\alpha_n^j}{n}]\\ 
0 & \hbox{ in } (\frac{j-1}{n}+\frac{\alpha_n^j}{n}, \frac{j}{n}] 
\end{array}
\right. 
\end{equation}
if $0<\alpha_n^j<1$, and $\hat u_n=u_n$ otherwise (see Fig.~\ref{f-liminf}). 
  \begin{figure}[h!]
\centerline{\includegraphics [width=5in]{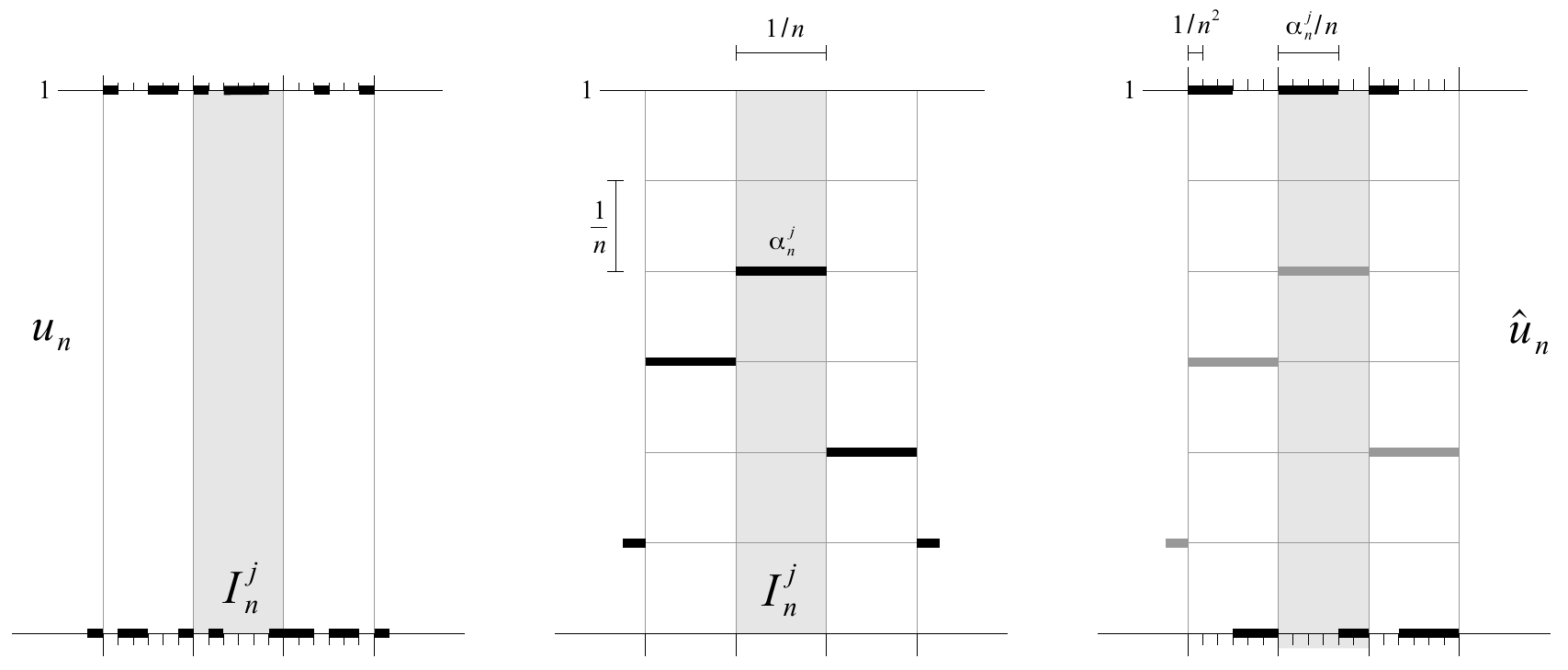}}
\caption{Construction of test functions for the lower bound}\label{f-liminf}
   \end{figure}

Note that in each $I_n^j$ this construction corresponds to setting $\hat{u}_n=1$ in the first $n\alpha_n^j$ points of 
$I_n^j\cap (\frac{1}{n^2}\mathbb N)$, and $0$ in the remaining $n- n\alpha_n^j$. By construction, it follows that $\hat{u}_n$ weakly* converges in $L^\infty(0,1)$ to the same limit $u$, and 
\begin{equation}\label{stima-hatun}
F_n(u_n)\geq F_n(\hat u_n).
\end{equation} 
Indeed, with regard to the nearest-neighbour connections, 
since $\hat u_n$ has at most one jump in $(\frac{j-1}{n}, \frac{j}{n})$, then, denoted by $S(v)$ the set of discontinuity points of a function $v$, $\#S(\hat u_n)=\min\{\#S(u_n),2\}$ in each interval $I_n^j$ for $j<n$, and $\#S(\hat u_n)=\min\{\#S(u_n),1\}$ in $I_n^n$. Concerning the long range interactions, for each pair of adjacent intervals $I_n^{j-1}$, $I_n^j$ the energy $F_n(u_n)$ counts the number of points in the symmetric difference 
$(n+K_{j-1}(u_n))\triangle  K_j(u_n)$, where $K_j(u)=\{i: u_i=1, \frac{i}{n^2}\in I_n^{j}\}$. The inequality (\ref{stima-hatun}) follows by noting that 
\begin{eqnarray*}
\# \big( (n+K_{j-1}(u_n))\triangle K_j(u_n)\big) &\geq&    \Big| \# (n+K_{j-1}(u_n)) - \# K_j(u_n)   \Big|\\
&=&   \# \big( (n+K_{j-1}(\hat u_n))\triangle  K_j(\hat u_n)\big).
\end{eqnarray*}
Now, we define the set 
$G_n=\bigcup_{j=1}^n R_n^j$ where 
$R_n^j=[\frac{j-1}{n}, \frac{j}{n}]\times[0, \alpha_n^j]$ for any $j=1,\ldots, n$.   
By construction, in each 
$Q_n(j,k)=(\frac{j-1}{n},\frac{j}{n}]\times(\frac{k-1}{n},\frac{k}{n}]$ 
we have $\chi_{G_n}=\hat u_n(\frac{j-1}{n}+\frac{k}{n^2})$. Hence, with fixed $\varrho>0$, for all 
 $n>\frac{1}{\varrho}$ 
\begin{eqnarray*}
F_n(\hat u_n)&\geq& \disp \frac{1}{n}\sum_{\{z,z^\prime\}\in\mathcal N(0,n]^2} \big((\chi_{G_n})_z-
(\chi_{G_n})_{z^\prime}\big)^2\\
&\geq& \disp\frac{1}{n}\sum_{\{z,z^\prime\}\in\mathcal N(nQ_\varrho^0)} \big((\chi_{G_n})_z-(\chi_{G_n})_{z^\prime}\big)^2
= E_n(\chi_{G_n}; Q_\varrho^0)
\end{eqnarray*}
Then, recalling Proposition \ref{nn2},  
since $E_n(\chi_{G_n}; Q_\varrho^0)=\mathcal H^1(\partial G_n \cap Q_\varrho^0)$, by compactness 
we deduce that (up to subsequences) $G_n$ converges in measure to a set of finite perimeter $G\subset Q_\varrho^0$.

In order to optimize the lower estimate, we also have to consider 
the part of the energy $F_n(\hat u_n)$ that comes from the interactions between $\frac{j}{n}$ and 
$\frac{j}{n}+\frac{1}{n^2}$ for $j=1,\dots, n-1$, corresponding to the length of 
the intersection of $\{y=0\}$ with the boundary of the periodic extension of $G_n$ to $(0,1)^2+(1/n,-1)$
in $(0,1)^2\cup ((0,1)^2+(1/n,-1)$. 
Setting $\tilde G_n=G_n\cup (G_n+(1/n,-1)$, 
for any $\eta\in (0,1)$ we have the estimate 
\begin{equation}\label{stima-en}
F_n(\hat u_n)\geq E_n(\chi_{\tilde G_n}; Q_\varrho^\eta).
\end{equation}
Note that $\tilde G_n$ 
converges in measure to $\tilde G=G\cup (G+(0,-1))$ 
in $Q_\varrho^0\cup (Q_\varrho^0+(0,-1))$.  
Since we can find $\eta\in(0,1)$ such that 
\begin{equation}
\mathcal H^1(\partial^\ast G\cap \{y=1-\eta\})=0,
\end{equation}  
then 
\begin{equation*} 
\mathcal H^1(\partial^\ast \tilde G\cap Q_\varrho^\eta)= 
\mathcal H^1(\partial^\ast G\cap Q_\varrho^0)+\mathcal H^1(\partial^\ast G\cap \{y=0\}
\cap Q_\varrho^0).
\end{equation*}

Hence, for such $\eta$, by applying the lower estimate for $E_n$ given by Proposition \ref{nn2} we get  
%
%
\begin{equation}\label{stima-nn2} 
\left. \begin{array}{ll}
\disp\liminf_{n\to+\infty} 
E_n(\chi_{\tilde G_n}; Q_\varrho^\eta)&\disp\geq  
\int_{\partial^\ast \tilde G\cap Q_\varrho^\eta} 
\|\nu_{\tilde G}\|_1\, d\mathcal H^1\\
&=
\disp\int_{\partial^\ast G\cap Q_\varrho^0}  
\|\nu_{G}\|_1\, d\mathcal H^1
+\mathcal H^1(\partial^\ast G\cap \{y=0\}\cap Q_\varrho^0).  
\end{array}
\right.
\end{equation}   

Now, we show that the weak$^\ast$ limit $u$ of $u_n$ belongs to $BV$, and $G$ is in fact its subgraph.
We start by considering $\overline u_n$ defined in each $I_n^j$ by setting 
$$\overline u_n=n\int_{I_n^j}\hat u_n(t)\, dt=n\int_{I_n^j}u_n(t)\, dt.$$
By construction, the set $G_n$ 
is the subgraph of $\overline u_n$; that is, $(x,y)\in G_n$ if and only if $0\leq y \leq \overline u_n(x)$.  
Moreover, $\{u_n\}$ is a Cauchy sequence in $L^1(\varrho,1-\varrho)$  
since $G_n$ converges in measure.  
By construction, $\overline u_n\weakstar u$ in $L^\infty(0,1)$, hence $\overline u_n\to u$ 
in $L^1(\varrho, 1-\varrho)$. Since $G_n$ is the subgraph 
of $\overline u_n$, the pointwise almost-everywhere convergence of $\overline u_n$ ensures that 
$(x,y)\in G$ if and only if $0\leq y \leq \overline u(x)$ 
almost everywhere in $Q_\varrho^0$.  
  
Since $G$ has finite perimeter, $u$ is a $BV$ function. 
Moreover, since $$\mathcal H^1(\partial^\ast G\cap \{y=0\}
\cap Q_\varrho^0)=\mathcal H^1(\{x\in (\varrho, 1-\varrho): 0<u(x)<1\}),$$ 
we have
\begin{eqnarray}\nonumber \label{stimma}
&&\hskip-2cm\disp\int_{\partial^\ast G\cap Q_\varrho^0}\|\nu_{G}\|_1\, d\mathcal H^1
+\mathcal H^1(\partial^\ast G\cap \{y=0\}\cap Q_\varrho^0)\\
&=&\mathcal H^1(\{x\in (\varrho, 1-\varrho): 0<u(x)<1\})+|Du|(\varrho, 1-\varrho)\nonumber\\&=&F(u;(\varrho, 1-\varrho)).
\end{eqnarray}  
Indeed, note that 
$$
\nu_G(x,y)= {(u',1)\over \sqrt{1+|u'|^2}} \quad x \hbox{-a.e.}\quad\hbox{ and }\quad
\nu_G(x,y)= (1,0){D_su\over |D_su|}\quad |D_su|\hbox{-a.e.}
$$
so that 
\begin{eqnarray*}
&&\hskip-2cm\int_{\partial^\ast G\cap Q_\varrho^0}\|\nu_{G}\|_1\, d\mathcal H^1
+\mathcal H^1(\{x\in(\varrho, 1-\varrho): u(x)\in\{0,1\}\})\\
&=&\int_\varrho^{1-\varrho} {1+|u'|\over \sqrt{1+|u'|^2}} \sqrt{1+|u'|^2}dx+
\int_{(\varrho,1-\varrho)}  |D_su|\\
&=&|Du|(\varrho,1-\varrho)+ \mathcal H^1(\varrho, 1-\varrho).
\end{eqnarray*}
Hence, by estimates (\ref{stima-hatun}), (\ref{stima-en}), (\ref{stima-nn2}) 
and (\ref{stimma}), we get the liminf inequality 
\begin{equation}\label{liminf}
\liminf_{n\to+\infty}F_n(u_n)\geq F(u;(\varrho, 1-\varrho))
\end{equation}
for any $\varrho>0$.
\end{proof}
\begin{proposition}[Upper bound]\label{limsup-prop}
Given $u\in BV(0,1)$ with values in $[0,1]$ there exists a recovery sequence 
$\{u_n\}$ such that 
$u_n\in\mathcal A_n$, $u_n\weakstar u$ in $L^\infty(0,1)$ and 
\begin{equation} \label{limsup-eq} 
\limsup_{n\to+\infty} F_n(u_n)\leq F(u).  
\end{equation} 
\end{proposition}
\begin{proof} 
As a first step, we show that it is sufficient to prove the limsup inequality 
for $u$ piecewise constant. 
To this end, we first use a mollification argument. We
choose $\eta>0$ such that both $\eta$ and $1-\eta$ 
are approximate continuity points for $u$, and we
extend $u$ by reflection in $(-\eta,0)$ and $(1,1+\eta)$; namely,
$$
u(x)= u(-x) \quad\hbox{ if }x\in(-\eta,0)\qquad \hbox{and }
\qquad 
u(x)= u(2-x) \quad\hbox{ if }x\in(1,1+\eta). 
$$
Denoting this extension by $u^\eta:(-\eta,1+\eta)\to {\mathbb R}$ it follows that 
\begin{equation}\label{stima-rifl}
|Du^\eta|(-\eta,1+\eta)\leq |Du|(0,1)+o(1)_{\eta\to 0}.
\end{equation}
Now, let $\{\varrho_\e\}$ be a sequence of smooth convolution kernels 
such that, for any $\e>0$, $\int_{\mathbb R}\varrho_\e(t)\, dt=1$ and supp$(\varrho_\e)=[-\e,\e]$; 
setting $u_\e^\eta=u_\eta\ast\varrho_\e$, for $\e<\eta$ we get the estimate \begin{equation}\label{stima-conv}
|Du_\e^\eta|(0,1)\leq |Du^\eta|(-\e,1-\e)\leq  |Du^\eta|(-\eta,1-\eta). 
\end{equation}
Since we need to estimate the measure of the set where the values of each function belong to $(0,1)$, 
we set for any $\delta\in(0,1)$ 
$$u_\e^{\eta,\delta}=\left(  \left(   \Big(u_\e^\eta-\frac{1}{2}\Big)(1+\delta)+\frac{1}{2}   \right)     \vee 0\right)\wedge 1;$$ 
with this definition, $0<u_\e^{\eta,\delta}<1$ implies $\frac{\delta}{4}\leq u_\e^\eta \leq 1-\frac{\delta}{4}$; then, the properties of the convolution ensure that 
\begin{equation}\label{misura}
\left. \begin{array}{ll}
 |\{x\in(0,1): 0<u_\e^{\eta,\delta}<1\}| &\leq |\{x\in (0,1): \frac{\delta}{4}\leq u_\e^\eta \leq1-\frac{\delta}{4}\}| \\
&\disp\leq |\{x\in (0,1): 0< u^\eta <1\}| +2\e.
\end{array}
\right. 
\end{equation}
Since 
$$|Du_\e^{\eta,\delta}|(0,1)\leq (1+\delta) |Du_\e^\eta|(0,1),$$ 
recalling (\ref{stima-rifl}), (\ref{stima-conv}) and (\ref{misura}) we get 
$$F(u_\e^{\eta,\delta})\leq (1+\delta) (F(u)+o(1)_{\eta\to 0})+2\e.$$  
Hence for any $\sigma>0$ we can find $\eta,\e$ and $\delta$ small enough to have 
$\|u-u_\e^{\eta,\delta}\|_{L^1}<\sigma$ and 
\begin{equation}\label{dens1} 
F(u_\e^{\eta,\delta})\leq F(u) +\sigma.
\end{equation}
Now we construct a sequence $\{u_k\}$ of piecewise-constant functions  
(where we omit the dependence on $\eta,\e$ and $\delta$) such that 
$u_k\to u_\e^{\eta,\delta}$ in $L^1(0,1)$ as $k\to+\infty$ and 
\begin{equation}\label{stima-scala}
F(u_k)\leq F(u_\e^{\eta,\delta}).
\end{equation} 
For any fixed integer $k\geq 1$ we consider the intervals 
$J_k^i=(\frac{i-1}{k}, \frac{i}{k}]$; since 
$u_\e^{\eta,\delta}\in C^0[0,1]$, for $k$ large enough either $J_k^i\cap \{ u_\e^{\eta,\delta}=0\}=\emptyset$  
or $J_k^i\cap \{ u_\e^{\eta,\delta}=1\}=\emptyset$ for any $i=1,\dots, k$. Hence, for such $k$ we can define $u_k$ by setting in each $J_k^i$
$$u_k=\left\{ \begin{array}{ll} 
0 & \mbox{ if } J_k^i\cap \{ u_\e^{\eta,\delta}=0\}\neq\emptyset\\
1 & \mbox{ if } J_k^i\cap \{ u_\e^{\eta,\delta}=1\}\neq\emptyset\\
u_\e^{\eta,\delta}(\frac{i}{k}) & \mbox{ if } J_k^i\subset \{ 0<u_\e^{\eta,\delta}<1\}.
\end{array}
\right. $$
Note that the uniform continuity of $u_\e^{\eta,\delta}$ ensures the uniform convergence of 
$u_k\to u_\e^{\eta,\delta}$, hence $u_k\to u_\e^{\eta,\delta}$ in $L^1(0,1)$. 
By construction we have that for $k$ large enough 
\begin{equation*}
\mathcal H^1(\{0<u_k<1\})\leq \mathcal H^1(\{0<u_\e^{\eta,\delta}<1\}), \qquad
|Du_k|(0,1)\leq |Du_\e^{\eta,\delta}|(0,1),
\end{equation*}  
hence (\ref{stima-scala}) holds and 
it is sufficient to prove the $\limsup$-inequality for $u$ piecewise constant
(see \cite{GCB} Remark 1.29). 

  \begin{figure}[h!]
\centerline{\includegraphics [width=6in]{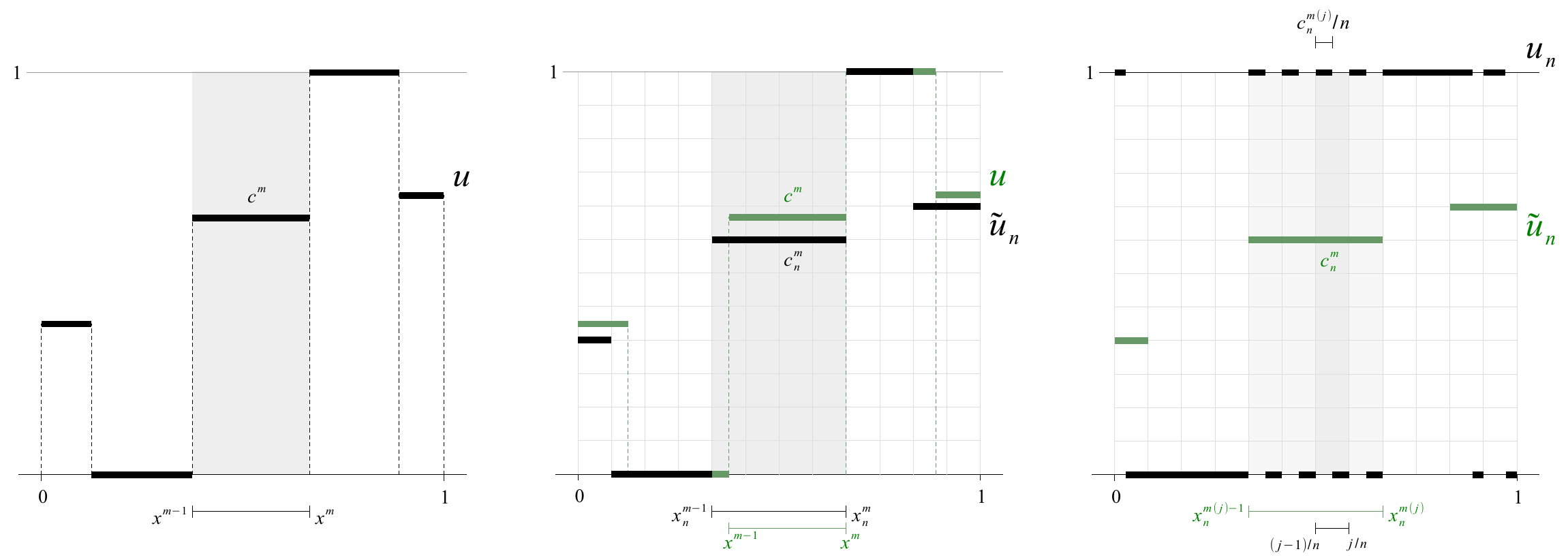}}
\caption{Construction of a recovery sequence}\label{limsup-fig}
   \end{figure} 
Let $u$ be a piecewise-constant function given by the partition $\{x^m\}_{m=0}^k$
and values $c^m\in[0,1]$ such that $u=c^m$ in $(x^{m-1},x^m)$ for $m=1,\ldots, k$  
(see the first picture in Fig.~\ref{limsup-fig}).
For each $m=0,\ldots k$ we set 
$x^m_n=\frac{\lfloor nx^m\rfloor}{n}$; we define 
$\tilde u_n$ by setting $\tilde u_n=c^m_n=\frac{\lfloor nc^m\rfloor}{n}$ 
in each interval $(x_n^{m-1}, x_n^m)$ (see the second picture in Fig.~\ref{limsup-fig}). 
In this way we get the inequality 
\begin{equation}\label{tildeu}
F(\tilde u_n)\leq F(u)+\frac{k}{n}.
\end{equation}
Now, we can construct the recovery sequence $u_n$ similarly as we did in (\ref{def-hatun}) in the proof of the liminf inequality. 
For each $j$ let $m(j)\in\{1,\dots, k\}$ be such that 
$(\frac{j-1}{n},\frac{j}{n})\subset (x^{m(j)-1}_n, x^{m(j)}_n)$; we define $u_n$ in  
$(\frac{j-1}{n},\frac{j}{n}]$ by setting 
$$ 
u_n=\left\{ \begin{array}{ll} 
1 & \mbox{ in } (\frac{j-1}{n},\frac{j-1}{n}+\frac{c_n^{m(j)}}{n}]\\
0 & \mbox{ in } (\frac{j-1}{n}+\frac{c_n^{m(j)}}{n}, \frac{j}{n}]
\end{array} 
\right.
$$
if $c_n^{m(j)}\neq 0,1$, $u_n=0$ if $c_n^{m(j)}=0$ and $u_n=1$ if $c_n^{m(j)}=1$
(see the third picture in Fig.~\ref{limsup-fig}).

Note that $u_n\weakstar u$ in $L^\infty(0,1)$. 
By construction $F_n(u_n)=F(\tilde u_n)$, hence 
\begin{equation} 
\limsup_{n\to+\infty} F_n(u_n)\leq F(u)
\end{equation}
after recalling (\ref{tildeu}).
\end{proof}

\section{Volume-constrained minimization problems} 

In this section we examine the behaviour of functionals $F_n$ subjected to the constraint of fixing the 
number of $i$ such that $u_i=1$. Since the form of the minimizers of such a constraint depends on the size of the domain, we extend the previous result to functions defined on $[0,L]$ for any $L>0$.

%

\subsection{Compatibility of the volume constraint}
Let $k_n$ be a sequence of integers such that $0\leq k_n\leq L n^2$ such that 
$$
\lim_{n\to+\infty}\frac{k_n}{L n^2}=\sigma\in(0,1).
$$ 
Let $L_n=\frac{\lfloor Ln^2\rfloor}{n^2}$ and $\mathcal A_n(L, k_n)$ be the set of the functions $u\colon(0, L_n]\to\{0,1\}$ with 
$u$ constant on each interval $(\frac{i-1}{n^2},\frac{i}{n^2}],$ $i=1,\ldots, \lfloor L n^2\rfloor$ 
 and such that
\begin{equation}\label{intc-n}
\#\Bigl\{i\in\{1,\ldots, \lfloor L n^2\rfloor\}: u_i=1\Bigr\}=k_n, 
\end{equation}
where $u_i$ stands for $u\big(\frac{i}{n^2}\big)$ as previously.  
Such a $u$ corresponds to a discrete function, which we still denote by $u$,
defined as its restriction to $
\big(\frac{1}{n^2}\mathbb N\big)\cap(0,L]$. 
If necessary we extend each such function to $0$ outside $(0, L_n]$. Such an extension
does not modify the energies we are going to consider and makes convergence statements easier.

We set  
\begin{equation}\label{def-en-L}
\mathcal E_n(L)=\Big\{ \{i,j\}: \ i, j \in \mathbb N \cap (0, L n^2] \hbox{ and } |i-j|= 1 \hbox{ or } |i-j|=n\Big\}.
\end{equation}
For $u\in \mathcal A_n(L,k_n)$ we define the functional 
\begin{equation}\label{def-fnL}
F_n^{L,k_n}(u)=\frac{1}{n}\sum_{\{i,j\}\in \mathcal E_n(L)} (u_i-u_j)^2.
\end{equation}

\begin{theorem}[compatibility of volume constraints]\label{th-con}
The sequence of functionals $\{F^{L,k_n}_n\}$ $\Gamma$-converges with respect to the weak-$*$ convergence in $L^\infty(0,L)$
to the functional $F^{L,\sigma}: L^\infty(0,L)\to [0,+\infty]$ given by 
\begin{equation}\label{def-fL}
F^{L,\sigma}(u)=\left\{\begin{array}{ll}
2\,\mathcal H^1(\{x\in (0,L): 
0<u(x)<1 \})+|Du|(0,L) &
\cr&\hskip-6cm \hbox{ if } u\in BV(0,L), \ 0\leq u\leq 1, \displaystyle \int_0^Lu\,dx=L\sigma   \cr
+\infty &\hskip-6cm  \hbox{ otherwise.}
\end{array}
\right.
\end{equation} 
\end{theorem}

\begin{proof} Since functions satisfying the integral constraint (\ref{intc-n}) converge to functions satisfying
$\int_0^Lu\,dx=L\sigma$, the liminf inequality is immediately satisfied. It remains to prove the existence
of a recovery sequence for $u$ in the domain of $F^{L,\sigma}$. By a density argument, it suffices to treat the 
case of $u$ piecewise constant since the construction in the proof of Proposition \ref{limsup-prop} is compatible with the integral constraint. Moreover, we may reduce to treat the case of $u$ constant, since the construction below is immediately extended to a piecewise-constant $u$. 

We now exhibit a recovery sequence for the constant target function $u=\sigma$ in $(0,L)$. 


Writing $k_n=(\lfloor Ln\rfloor +1)a_n+b_n$ with $a_n\in\{0,\dots, n \}$ and 
$b_n\in\{0,\dots,  \lfloor Ln\rfloor\}$, 
we construct a function $v_n$ defined in $(0,\frac{\lfloor Ln\rfloor + 1}{n}]$ as follows. 
We denote by $\lambda_n$ the number 
$\lfloor Ln^2\rfloor -n\lfloor Ln\rfloor$ 
(measuring the `defect of periodicity' of the interval $[0,L]$)
and define $v_n$  as follows.

$\bullet$ If $a_n\leq \lambda_n $ then we have two constructions, according to $b_n$:
 in each $(\frac{j-1}{n},\frac{j}{n}]$  with  $b_n\geq j$ we set
$$ 
v_n=\left\{ \begin{array}{ll} 
1 & \mbox{ in } (\frac{j-1}{n},\frac{j-1}{n}+\frac{a_n+1}{n^2}]\\
0 & \mbox{ in } (\frac{j-1}{n}+\frac{a_n+1}{n^2}, \frac{j}{n}];
\end{array}  \right.
$$ 
 in each $(\frac{j-1}{n},\frac{j}{n}]$ with $b_n< j$ we set
$$ 
v_n=\left\{ \begin{array}{ll} 
1 & \mbox{ in } (\frac{j-1}{n},\frac{j-1}{n}+\frac{a_n}{n^2}]\\
0 & \mbox{ in } (\frac{j-1}{n}+\frac{a_n}{n^2}, \frac{j}{n}];
\end{array} \right.
$$

$\bullet$ If $\lambda_n<a_n<n$, 
then we write $a_n-\lambda_n+b_n=\gamma_n \lfloor Ln\rfloor +\delta_n,$ 
with $\delta_n<\lfloor Ln\rfloor$, and the constructions are as follows: if $\delta_n\geq j$ we set 
$$ 
v_n=\left\{ \begin{array}{ll} 
1 & \mbox{ in } (\frac{j-1}{n},\frac{j-1}{n}+\frac{a_n+\gamma_n+1}{n^2}]\\
0 & \mbox{ in } (\frac{j-1}{n}+\frac{a_n+\gamma_n+1}{n^2}, \frac{j}{n}]
\end{array} 
\right.$$
and  if  $\delta_n<j$
$$ 
v_n=\left\{ \begin{array}{ll} 
1 & \mbox{ in } (\frac{j-1}{n},\frac{j-1}{n}+\frac{a_n+\gamma_n}{n^2}]\\
0 & \mbox{ in } (\frac{j-1}{n}+\frac{a_n+\gamma_n}{n^2}, \frac{j}{n}]
\end{array} 
\right.
$$

Now, we define $u_n$ as the restriction of $v_n$ to $(0,L_n]$. 
Note that $\#\{i \in\{1,\ldots, \lfloor L n^2\rfloor \}: (u_n)_i=1\}=k_n$, that is $u_n\in\mathcal A(L,k_n)$.

Since $a_n\neq 0$ for $n$ large enough, 
the number of $\{i,j\}$ such that $|i-j|=1$ and $(u_n)_i\neq (u_n)_j$ is bounded by 
$2(\lfloor Ln\rfloor +1)$, and the number of jumps between points at distance $n$ is at most $1$. Hence 
%
%
\begin{equation}
F_n^{L,k_n}(u_n)\leq\frac{1}{n} (2\lfloor Ln\rfloor + 3)=2L+o(1)_{n\to+\infty}=F^{L,\sigma}(u)+o(1)_{n\to+\infty}. 
\end{equation}
\end{proof}

\subsection{A size effect}\label{sizeffect}
As a consequence of Theorem \ref{th-con}, minimum problems for $F_n^{L, k_n}$ (i.e., volume-constrained minimum problems for $F_n$ on $[0,L]$) converge to the minimum problem
\begin{eqnarray}\nonumber
&&\min\Bigl\{2\,\mathcal H^1(\{x\in (0,L): 0<u(x)<1 \})+|Du|(0,L):\\
&&\hskip4cm u\in BV((0,L);[0,1]), \int_0^Lu\,dx=\sigma L\Bigr\}.
\end{eqnarray}
 
This problem can be simplified by remarking that 

$\bullet$ if $u$ takes the values $1$ and $0$ on a set of non-zero measure then $\mathcal H^1(\{x\in (0,L): 0<u(x)<1 \})=0$. Indeed if $u$ takes the values $1$ and $0$ then $|Du|\ge 1$, and the function $v(x)=\chi_{[0,\sigma L]}$ has a lower energy than $u$;

$\bullet$ we may assume that $\{x\in (0,L): 0<u(x)<1 \}$ is an interval and $u$ is constant on $\{x\in (0,L): 0<u(x)<1\}$.

\noindent We end up with four cases:

\begin{enumerate}
\item[A)] the minimum is $u_{\rm min}=\sigma$. In this case, $F^{L,\sigma}(u_{\rm min})= 2L$;
\item[B)] the minimum is $u_{\rm min}=\chi_{[0,\sigma L]}$. In this case, $F^{L,\sigma}(u_{\rm min})= 1$;
\item[C)] the minimum is obtained by minimizing on functions of the form $u=
{\sigma L\over y}\chi_{[0,y]}$. In this case, $F^{L,\sigma}(u_{\rm min})= 2\sqrt{2\sigma L}$;
\item[D)] the minimum is obtained by minimizing on functions of the form $u=
1-{(1-\sigma) L\over y}\chi_{[0,\sigma y]}$. In this case, $F^{L,\sigma}(u_{\rm min})= 2\sqrt{2(1-\sigma) L}$. 
\end{enumerate}

  \begin{figure}[h!]
\centerline{\includegraphics [width=5in]{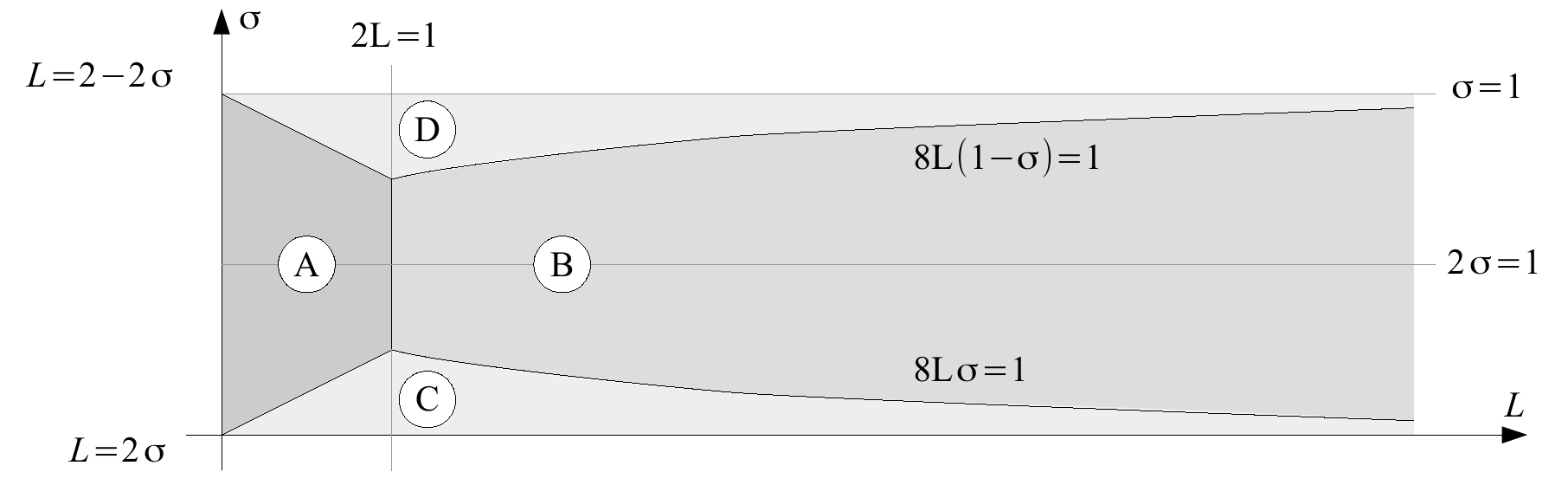}}
\caption{Description of minimizers}\label{sizefig}
   \end{figure}
\smallskip
In addition, note that the behaviour of the minimum problems is the same for $\sigma$ and $1-\sigma$, so that it is sufficient to examine the case $\sigma\le 1/2$. An explicit calculation yields the analysis highlighted in Fig.~\ref{sizefig}.

\begin{remark}[size effect for volume-constrained minimization]\rm
As remarked above, it is not restrictive to limit our description to $\sigma\le 1/2$.
We have two different behaviours depending on $L$.

\begin{figure}[h!]
\centerline{\includegraphics [width=4in]{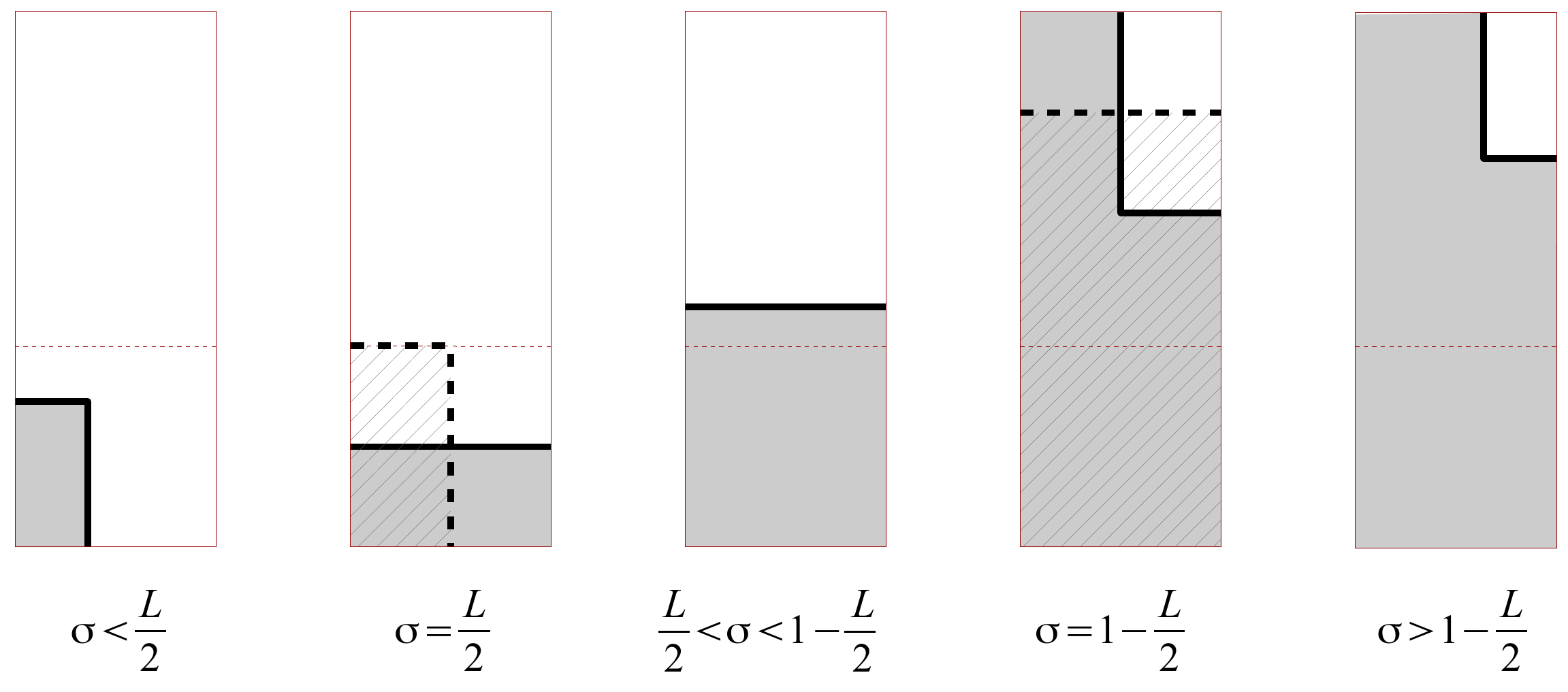}}
\caption{Evolution of the form of the minimizers in small domains}\label{piccoloNP}
   \end{figure}

\bigskip
{\bf Small-size domain:} if $L<{1\over 2}$ then we have

i) for $\sigma<{L\over 2}$ the minimizers are
$$u_1= \sqrt{2\sigma L} \,\chi_{[0,{1\over 2}\sqrt{2\sigma L}]},\qquad
u_2= \sqrt{2\sigma L}\, \chi_{[L-{1\over 2}\sqrt{2\sigma L},L]};
$$

ii) for $\sigma ={L\over 2}$ the minimizers
are the constant $u=\sigma$ and the two functions
$$
u_1= L \,\chi_{[0,{L/2}]},\qquad
u_2= L\, \chi_{[L/2,L]};
$$

iii) for ${L\over 2}<\sigma $ the minimizer
is the constant $u=\sigma$.

In Fig.~\ref{piccoloNP} we picture two-dimensional sets corresponding to minimizers of
the energy at varying $\sigma$. The length of the part of the sets highlighted in the figure 
gives the corresponding value of the energy.
Note that at $\sigma={L\over 2}$ (and symmetrically at $\sigma=1-{L\over 2}$, we have a discontinuity in the form of minimizers.

\bigskip

{\bf Large-size domain:} if $L>{1\over 2}$ then we have

i) for $\sigma<{1\over 8L}$ the minimizers are as in case {(}C) above
$$u_1= \sqrt{2\sigma L} \,\chi_{[0,{1\over 2}\sqrt{2\sigma L}]},\qquad
u_2= \sqrt{2\sigma L}\, \chi_{[L-{1\over2}\sqrt{2\sigma L},L]};
$$

ii) for $\sigma ={1\over 8L}$ the four minimizers
are the functions
$$
u_1= {1\over 2} \,\chi_{[0,{1\over4}]},\qquad
u_2= {1\over 2}\, \chi_{[L-{1\over4},L]},\qquad
u_3= \chi_{[0,{1\over 8}]},\qquad
u_4= \chi_{[L-{1\over 8},L]};
$$

iii) for $\sigma>{1\over 8L} $ the minimizers are
$$
u_1= \chi_{[0,\sigma L]},\qquad
u_2= \chi_{[(1-\sigma)L,L]}.
$$

In Fig.~\ref{grandeNP} we again picture minimizers at varying $\sigma$.
Again, note the discontinuity in the form of the minimizers
at $\sigma={1\over 8L}$  (and symmetrically at $\sigma=1-{1\over 8L}$). 

\begin{figure}[h!]
\centerline{\includegraphics [width=5in]{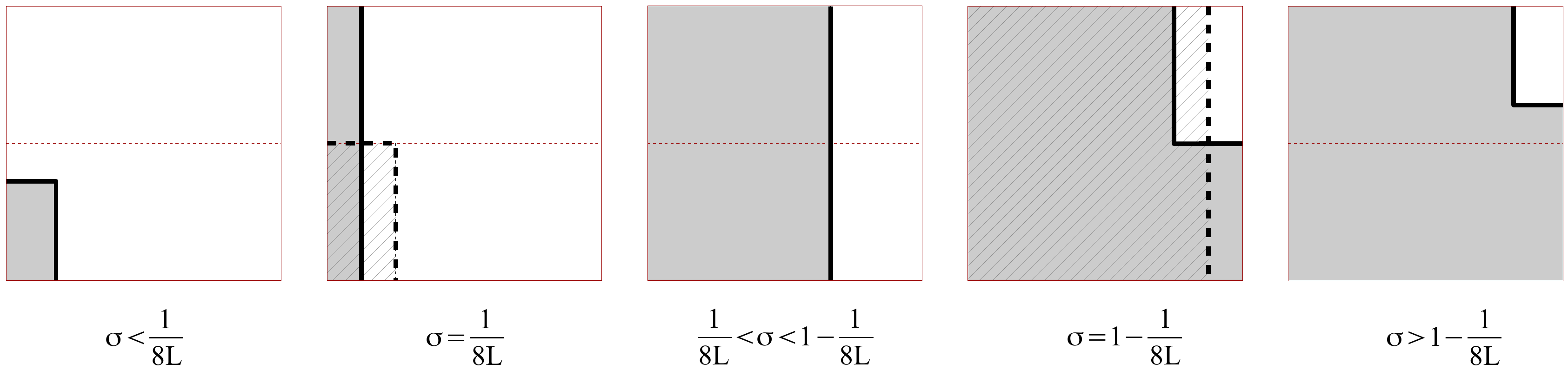}}
\caption{Evolution of the form of the minimizers in large domains}\label{grandeNP}
   \end{figure}



\end{remark}

\begin{remark}\rm  
From the description of minimizers in the previous remark we easily derive the shape of minimizers for the discrete problems, given by the corresponding recovery sequences. A pictorial description for small-size domains in given in Fig.~\ref{recsequence}.
\end{remark}

\section{Boundary effects in periodic minimization}
In this section we analyze the effect of periodic boundary conditions on the
volume-constrained minimization. The overall behaviour in dependence 
of $L$ and $\sigma$ will be described by introducing an additional parameter $\tau$,
which quantifies the `defect of ${1\over n}$-periodicity' of the underlying domain
$[0,L]$, and computing a $\tau$-depending $\Gamma$-limit. The existence
of the $\Gamma$-limit only up to subsequences depending on $\tau$ is not an 
uncommon feature when studying fine effects of homogenization depending
on boundary effects (see e.g.~\cite{BT} Section 1.3).

We define the set of the periodic interactions as 
$$\mathcal E_n^{\#}(L)  
= \{ \{i,j\}: i,j \in\mathbb N\cap (0, \lfloor L n^2\rfloor] : 
|i-j|\in\{1, \lfloor L n^2\rfloor -1, n, \lfloor L n^2\rfloor-n\}\}.$$
For $u\in\mathcal A_n(L,k_n)$  (as previously defined)
we set 
$$F_n^{\#}(u)=\frac{1}{n}\sum_{\{i,j\}\in \mathcal E_n^{\#}(L)} (u_i-u_j)^2$$
(the dependences on $L$ and $k_n$ are omitted).

We consider the sequence $\lambda_n=\lfloor Ln^2\rfloor - n\lfloor Ln\rfloor\in [0,n)$. 
Let $\tau_{n_m}$ be a subsequence of $\tau_n=\frac{\lambda_{n}}{n}$ converging to $\tau\in [0,1]$; we compute the $\Gamma$-limit 
for the corresponding sequence $\{F^\#_{n_m}\},$ again denoted by $\{F^\#_n\}$. 

\begin{remark}\label{rem-tau} \rm 
We denote by $\tau(L)$ the set of the limits of converging subsequences 
of $\{\tau_n\}$. If $L=\frac{p}{q}$, $p,q\in \mathbb N$ co-primes, then 
$\tau(L)=\{\frac{k}{q}: k\in\mathbb N, k\leq q\}$. Otherwise, if $L\not\in \mathbb Q$ then 
$\tau(L)=[0,1]$. Indeed, the difference $\tau_n-(Ln-\lfloor Ln\rfloor)$ goes to $0$ as $n\to+\infty$, and a theorem of Kronecker ensures that the sequence of the fractional parts of $Ln$ is dense in $[0,1]$ if $L$ is not rational. Note that this implies that for any $\tau\in[0,1]$ the set of the values $L>0$ such that 
$\tau\in\tau(L)$ is dense in $(0,+\infty)$.
\end{remark}

%

\begin{figure}[h!]
\centerline{\includegraphics [width=3in]{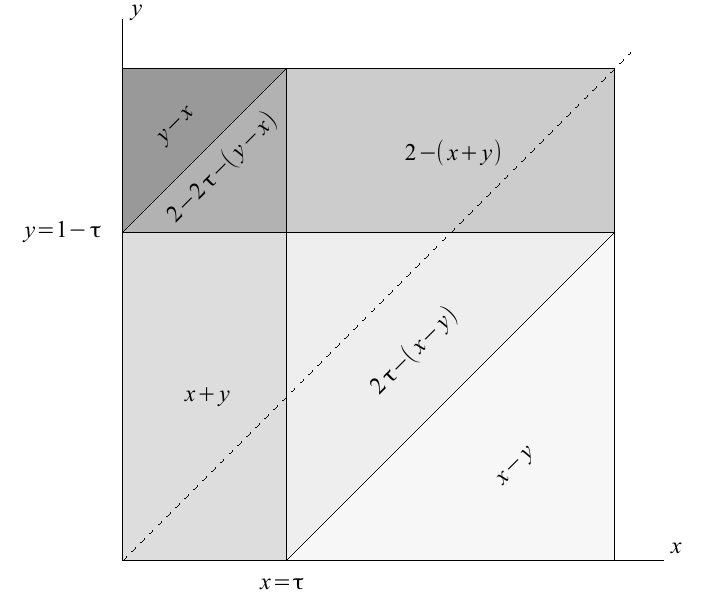}}
\caption{Pictorial description and values of $\phi_\tau$}\label{figura-phi}
   \end{figure}
For $x,y\in [0,1]$ we define 
\begin{equation}\label{fitau}
\phi_\tau(x, y)= {\mathcal H}^1([0,1] \cap (U_x^\tau\triangle U_y))
\end{equation}
 with $U_y=[0,y]$ and $U_x^\tau=[-\tau,x-\tau]\cup[1-\tau, x+1-\tau ]$ (see Fig.~\ref{figura-phi}).

\begin{theorem} Let $\tau$ be defined as above, and $\phi_\tau$ the corresponding function in {\rm(\ref{fitau})}.
The sequence of functionals $\{F_n^{\#}\}$ $\Gamma$-converges with respect to the weak-$*$ convergence in $L^\infty(0,L)$ 
to the functional $F_\tau^{\#}\colon L^\infty(0,L)\to [0,+\infty]$ given by 
\begin{equation}\label{def-fper}
F_\tau^{\#}(u)=\left\{\begin{array}{ll}
2\,{\mathcal H}^1(\{x\in(0,L) : 0<u<1\})+|Du|(0,L)+\phi_\tau(u(0^+), u(L^-))\\
& \hskip-7cm\hbox{ if } u\in BV(0,L), \ 0\leq u\leq 1,  \ \displaystyle\int_0^Lu\,dx=L\sigma  \\ 
+\infty & \hskip-7cm \hbox{ otherwise.}
\end{array}
\right.
\end{equation} 
\end{theorem}

  \begin{figure}[h!]
\centerline{\includegraphics [width=5.3in]{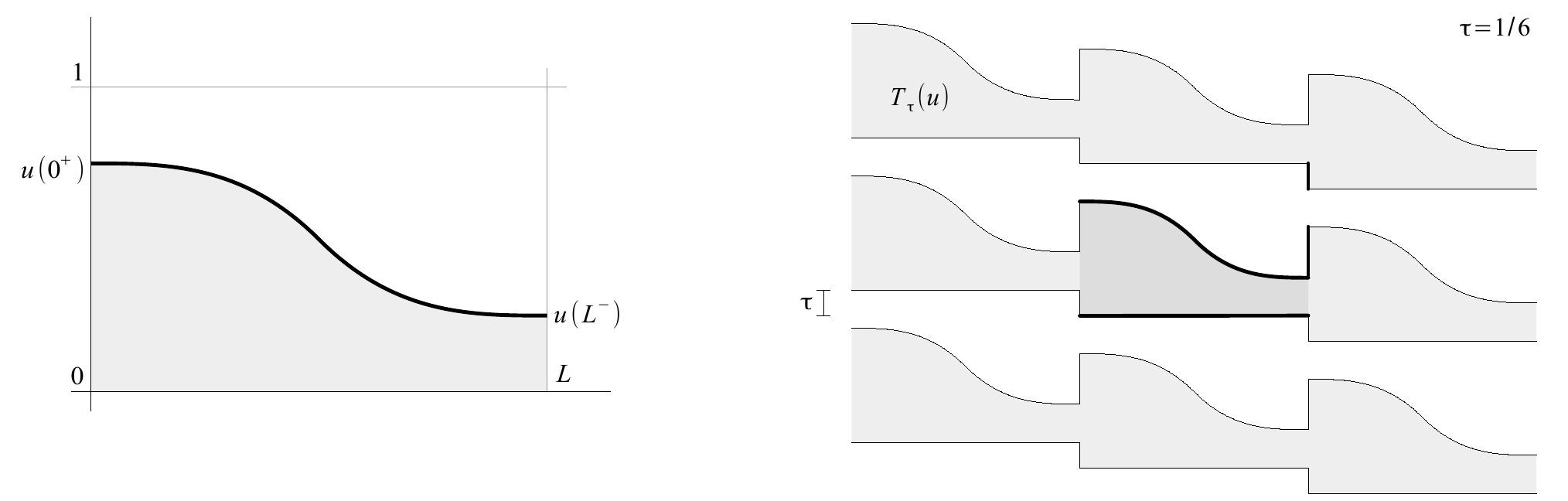}}
\caption{Description of $T_\tau(u)$}\label{T-tau}
   \end{figure}
\begin{remark}[Interpretation in terms of perimeters]\rm
Given $u\in BV((0,L); [0,1])$ and $\tau\in[0,1]$, 
we define the set $T_\tau(u)\subset\mathbb R^2$ as 
\begin{equation}\label{def-ttau}
T_\tau(u)=\bigcup_{n_1,n_2\in\mathbb Z} \big(V(u)+ (n_1 L, n_1 (1-\tau)+n_2)\big),
\end{equation}
where $V(u)=\{(x,y): 0\leq x \leq L, 0\leq y\leq u(x)\}$ is the subgraph of $u$ 
(see Fig.~\ref{T-tau}). 
Note that the value of the $1$-perimeter of $T_\tau(u)$ in a periodicity cell 
equals $F_\tau^{\#}(u)$; i.e.,
\begin{eqnarray*}
\hbox{\rm Per}_1(T_\tau(u); (0,L]\times(0,1])&=&2{\mathcal H}^1(\{x\in (0,L): 0<u<1\})\\&&\qquad + |Du|(0,L)
+\phi_\tau(u(0^+), u(L^-)).
\end{eqnarray*}
The value of $\phi_\tau(u(0^+), u(L^-))$ is the measure of the part of the graph 
on the boundary of the periodicity cell highlighted in Fig.~\ref{T-tau}.
\end{remark}

\begin{remark}[Properties of $\phi_\tau$]\label{prop-fi}\rm
If $\tau=0$, then $\phi_0(x, y)=|x-y|$. 
Since $\phi_{\tau}(x,y)=\phi_{1-\tau}(y,x)$, we can reduce to consider the case 
$\tau\leq\frac{1}{2}$. The following monotonicity property of $\phi_\tau$ will be useful in the computation of the minima of the functional $F^\#$: 
\begin{equation}\label{monotonia-phi} 
\phi_\tau(x,y)\leq \phi_\tau(y,x) \ \ \hbox{ if }\ \ x\geq y  \ \ \hbox{ for }\ \ \tau\leq \frac{1}{2}.  
\end{equation}
Moreover, note that 
$$\phi_\tau(s,s)=\left\{ 
\begin{array}{ll}
2s & \hbox{ if } s\leq \min\{\tau,1-\tau\}\\
2\min\{\tau,1-\tau\} & \hbox{ if } \min\{\tau,1-\tau\}\leq s \leq \max\{\tau,1-\tau\}\\
2(1-s) & \hbox{ if } \max\{\tau,1-\tau\}\leq s, 
\end{array}
\right. $$
hence for the constant function $u=\sigma$ we get $F_\tau^{\#}(u)=2\min\{\tau,1-\tau,\sigma,1-\sigma\}.$ 
\end{remark}

%
%
%
%

\begin{proof}[Proof of the lower bound] 
Let $\{u_n\}$ be such that $u_n\in \mathcal A_n(L,k_n)$ and $F^\#_n(u_n) \leq c<+\infty.$ 
Since $\{u_n\}$ is equibounded in $L^\infty(0,L)$ we can assume that (up to subsequences) 
$u_n\weakstar u$ in $L^\infty(0,L)$. 
Following the proof of Proposition \ref{proof-liminf}, 
we denote by $\alpha_n^j$ the integral mean of $u_n$ in $I_n^j$, where $I_n^j=(\frac{j-1}{n}, \frac{j}{n}]$ for $j<\lfloor Ln \rfloor$ and $I_n^{\lfloor Ln \rfloor}=(\frac{\lfloor Ln \rfloor}{n}, L_n]$. 
Firstly, we 
define $\check{u}_n$ in each $I_n^j$ with $j>1$ by setting 
\begin{equation}\label{def-hatcheck}
\check u_n= \left\{ \begin{array}{ll}
1 & \hbox{ in } (\frac{j-1}{n}, \frac{j-1}{n}+\frac{\alpha_n^j}{n}]\\ 
0 & \hbox{ in } (\frac{j-1}{n}+\frac{\alpha_n^j}{n}, \frac{j}{n}] 
\end{array}
\right. 
\end{equation}
if $0<\alpha_n^j<1$, and $\check u_n=u_n$ otherwise.  Note that this part of the construction the same as the one in definition (\ref{def-hatun}).  
Now, we define $\check u_n$ in $I_n^1$ in order to minimize the $n$-range jumps in this interval. To this end, 
we have to set, whenever possible, $(\check u_n)_i=1$ if $i$ belongs to the set $\mathcal I^1$ of the indices such that 
$(\check u_n)_{i+n}= (\check u_n)_{i+\lfloor Ln^2\rfloor-n}=1$ and $ (\check u_n)_i=0$ if  $i$ belongs to the set $\mathcal I^0$ of the indices such that 
$(\check u_n)_{i+n}=(\check u_n)_{i+\lfloor Ln^2\rfloor-n}=0$. 
Note that, by construction, $\mathcal I^1$ has one of the three following forms:  $\mathbb Z\cap (0,i^{\prime}]$ with $i^\prime\leq n\tau_n$, 
$\mathbb Z\cap (n\tau_n, i^{\prime\prime}]$ with $n\tau_n<i^{\prime\prime}\leq n$ or $\mathbb Z\cap ( (0,i^{\prime}]\cup (n\tau_n, i^{\prime\prime}])$. Similarly, the set $\mathcal I^0$ can be written as the union of at most two ``intervals'' 
$\mathbb Z\cap (j^\prime,n\tau_n]$ and  
$\mathbb Z\cap (j^{\prime\prime}, n]$.  
We define $\check u_n$ in $I_n^1$ by considering three cases. 
If $n\alpha_n^1\leq \# \mathcal I^1$, we set $(\check u_n)_i=1$ 
for the first $n\alpha_n^1$ indices $i$ in $\mathcal I^1$, and $0$ otherwise. 
If $ \# \mathcal I^1<n\alpha_n^1\leq n-\#\mathcal I^0$, then we define  
$(\check u_n)_i=1$ for any $i\in\mathcal I^1$ and for the first $n\alpha_n^1-\#\mathcal I^1$ points in the complementary set of $\mathcal I^1\cup \mathcal I^0$, 
and $0$ otherwise. 
Finally, if $n-\#\mathcal I^0<n\alpha_n^1$, we define $(\check u_n)_i=1$ 
for any $i$ in the complementary set of $\mathcal I^0$ and in the first 
$n\alpha_n^1-(n-\#\mathcal I^0)$ points of $\mathcal I^0$, and $0$ otherwise. 
The function $\check u_n$ belongs to $\mathcal A_n(L,k_n)$, 
and, following the idea of the proof of the estimate (\ref{stima-hatun}),  
from the construction of $\check u_n$ we deduce that
\begin{equation}\label{stima-checkun} 
F^\#_n(u_n)\geq F^\#_n(\check u_n)-\frac{c}{n}, 
\end{equation}
where $c$ is independent on $n$. 

As before, we now construct a set $\tilde G_n\subset\mathbb R^2$ such that for $\eta>0$ we have 
$$F^\#_n(\check u_n)\geq E_n(\tilde G_n; Q^\eta),$$
where $E_n$ is the functional defined in (\ref{def-en}) and $Q^\eta=(\eta,L+\eta)\times(\eta,1+\eta)$. Denoting by $\check u_n^\#$ the periodic extension to $\mathbb R$ of $\check u_n$, we set 
$$\chi_{G_n}=\check u_n^\#\Bigl(\frac{j-1}{n}+\frac{k}{n^2}\Bigr) \ \ \hbox{ in } \ \  Q_n^{k,j}=\Bigl(\frac{j-1}{n},\frac{j}{n}\Bigr]\times\Bigl(\frac{k-1}{n},\frac{k}{n}\Bigr]$$
with $j=1,\dots, \lfloor Ln \rfloor$ and $k=1, \dots, n$. We set 
$$\tilde G_n = \bigcup_{n_1,n_2\in \mathbb Z} \Bigl(G_n+\Bigl(n_1\frac{\lfloor Ln\rfloor}{n}, n_1(1-\tau_n)+n_2 \Bigr)\Bigr).$$ 
The same argument as in the proof of Proposition \ref{proof-liminf} ensures that $G_n$ converges in measure to a set $G$ with finite perimeter in $Q^0$, which is in fact the subgraph $V(u)$ of the weak$^\ast$-limit $u$ of $u_n$. 
Note that $u$ turns out to be $BV(0,L)$ with values in $[0,1]$, 
and $\int_0^L u\, dx=\sigma L$. Let $\eta\in (0,1)$ be such that 
$$\mathcal H^1(\partial^\ast T_\tau(u)\cap \partial Q^\eta)=0,$$ 
where $T_\tau$ is defined in (\ref{def-ttau});  
hence, by Proposition \ref{nn2} 
\begin{eqnarray*}
\liminf_{n\to +\infty} F^\#_n(u_n)&\geq& \liminf_{n\to +\infty} F^\#_n(\check u_n)
\ge  \liminf_{n\to +\infty} E_n(\tilde G_n; Q^\eta)\\&\geq  &
\hbox{Per}_1(T_\tau(u); Q^\eta)=\hbox{Per}_1(T_\tau(u); (0,1]\times (0,L])
\end{eqnarray*}
as claimed. \end{proof}

%
%
%
%
%
%
%
%
%
%

\begin{proof}[Proof of the upper bound] 
We can use the same recovery sequence as in the proof of the upper bound in Theorem \ref{th-con}. Indeed, the approximation used in the proof of Proposition \ref{limsup-prop} is compatible with the addition of the boundary term $\phi_\tau$, which is continuous along the sequences constructed therein. Moreover, the additional interactions taken into account in $F_n^\#$ asymptotically give the 
term with $\phi_\tau$.
\end{proof}

\subsection{Analysis of minimum problems}

We now briefly describe the behaviour of minimum problems for $F^\#_\tau$ in dependence of $\tau$, $\sigma$ and $L$. In order to understand the behaviour of minimizers, it is convenient to refer to the two-dimensional interpretation of the energies, where the effect of the mismatch in periodicity can be seen as the necessity to consider the extension of subsets on $(0,L)\times(0,1)$ by periodicity on the Bravais lattice generated by $(L,1-\tau)$ and $(0,1)$, as in the definition of $T_\tau$ above. Note that this extension has no significant energetic effect for sets as the ones in Fig.~\ref{grandeNP}, but it might for sets as in Fig.~\ref{piccoloNP}, in particular for rectangles corresponding to constant $u$. This will lead to a more complex typology of minimizers.

We first note that in order to compute minimum values, we can always reduce to piecewise-constant functions, as in the previous analysis of minimum values of $F^{L,\sigma}$.  
Note moreover that, setting $\tilde u(x)=u(L-x)$, we have 
$$F_\tau^\#(u)=F_\tau^\#(1-\tilde u).$$
Hence, we may limit the study to the case $\sigma\leq \frac{1}{2}$. 
Recalling that $F_\tau(u)=F_{1-\tau}(\tilde u)$ (see Remark \ref{prop-fi})
we can also assume $\tau\leq \frac{1}{2}$.

We start by showing that we can assume $u$ monotone non-increasing. 
Indeed, let $u$ be a piecewise-constant function in $[0,L]$ 
and denote by $u_{d}$ the non-increasing rearrangement of $u$.
Let $z_m=\min\{z : u(z)=m=\min u\}$ and $z_M=\min\{z : u(z)=M=\max u\}$. 
If $z_m\geq z_M$, then  
\begin{equation*}
F^{\#}_\tau(u)-F^{\#}_\tau(u_{d})\geq 
u(L)-u(0)+\phi_\tau(u(0), u(L))
-(m-M+\phi_\tau(M, m)).
\end{equation*}
Since for $\tau\leq \frac{1}{2}$ the function $y-x+\phi_\tau(x,y)$ turns out to be non-increasing along any direction 
$(h,-k)$ with $h,k\geq 0$, 
we deduce that 
\begin{equation*}
u(L)-u(0)+\phi_\tau(u(0),u(L)) \geq m-M+\phi_\tau(M,m)
\end{equation*}
showing that $F^{\#}_\tau(u)\geq F^{\#}_\tau(u_{d}).$ 
If $z_m<z_M$, the conclusion follows by recalling (\ref{monotonia-phi}) and noting that for $\tau\leq \frac{1}{2}$ 
the function $x-y+\phi_\tau(x,y)$ is non-increasing along any direction 
$(-h,k)$ with $h,k\geq 0$. Indeed  
\begin{eqnarray*}
F^{\#}_\tau(u)-F^{\#}_\tau(u_{d})&\geq& 
u(0)-u(L)+\phi_\tau(u(0),u(L))-(m-M+\phi_\tau(M, m)\\
&\geq&u(0)-u(L)+\phi_\tau(u(0),u(L))-(m-M+\phi_\tau(m, M))\ \geq\ 0.
\end{eqnarray*}
%
Hence, we may assume that $u$ has the form $u(0)\chi_{(0,y)}+u(L)\chi_{(y,L)}$ 
with $u(0)\geq u(L)$ (and with the integral constraint $u(0)y+u(L)(L-y)=\sigma$).  
If $u(0)<1$ and $0<u(L)$, then the monotonicity of $x-y+\phi_\tau(x,y)$ ensures that 
$$F_\tau^\#(u)\geq F_\tau^\#(u_\sigma)$$
where $u_\sigma$ is the constant function with value $\sigma$.  
Moreover, $F_\tau^\#(u)=F_\tau^\#(u_\sigma)$ if and only if $\tau\leq \sigma
$ and 
$u(0)-u(L)\leq\tau.$ 
  \begin{figure}[h!]
\centerline{\includegraphics [width=5in]{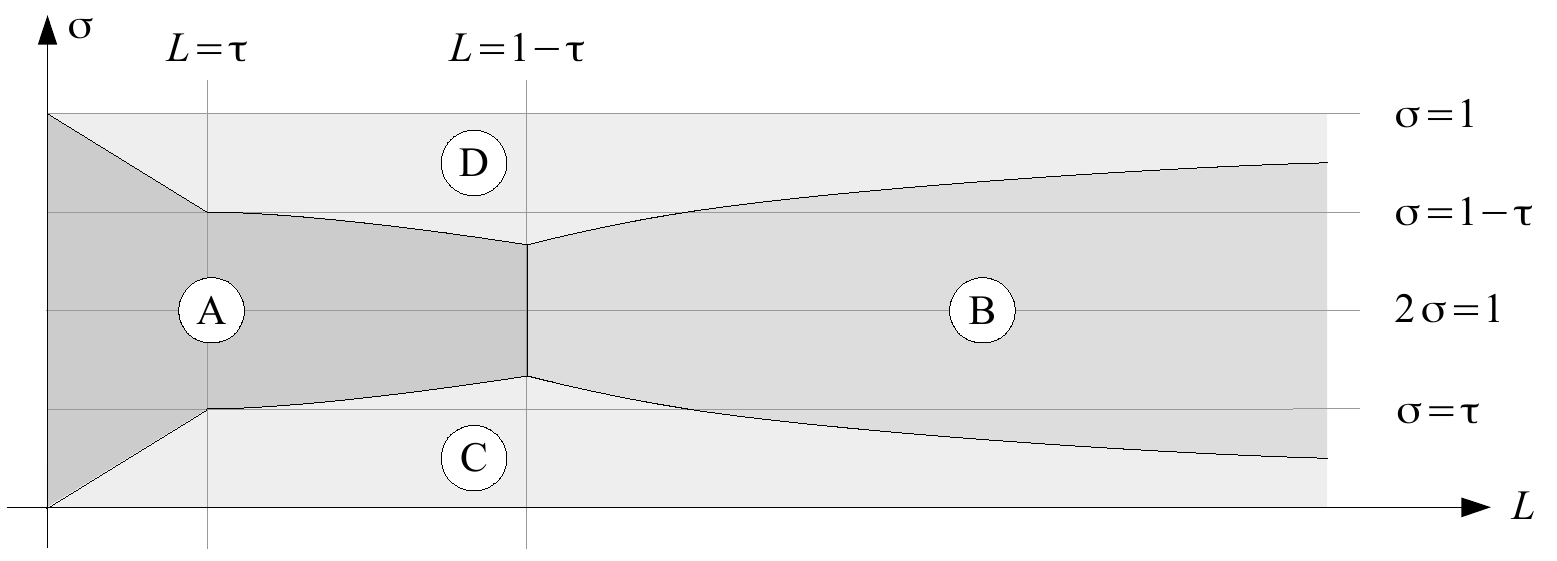}}
\caption{Description of minimum problems at given $\tau\in(0,1/2)$}\label{teletubbie}
   \end{figure}
Concluding, we again end up with four cases, pictured in Fig.~\ref{teletubbie}.
In order to take into account all cases with a common notation, we set
$$
\tau_*=\min\{\tau, 1-\tau\},\qquad \tau^*=\max\{\tau, 1-\tau\}.
$$

A) If $(L,\sigma)$ satisfies one of the following conditions   
\begin{equation*} 
\left. 
\begin{array}{ll}\vspace{3mm}
\displaystyle L\le \tau_* &\displaystyle\qquad \hbox{ and } \qquad L\leq \sigma\leq 1- L\\
\displaystyle\tau_*\le L\le \tau^* 
&\displaystyle\qquad \hbox{ and } \qquad 
{(L+\tau_*)^2\over 4L}\le \sigma\le 1-{(L+\tau_*)^2\over 4L},  
\end{array}
\right. 
\end{equation*}
then a minimizer is $u_{\min}=\sigma$, with energy $F_\tau^\#(u_{\min})=2L+2\min\{\sigma,1-\sigma, \tau_*\}$. 
Note that if $\tau=0$ the conditions for $(L,\sigma)$ are described by $L\leq 1$ and 
$\frac{1}{4L}\leq \sigma\leq 1-\frac{1}{4L}$, similarly to the situation pictured in Fig.~\ref{sizefig}. 
The minimizer is unique only if $\sigma<\tau_\ast$ or $\sigma>\tau^\ast$ (see Remark \ref{rem-minimi}); 


B) if $L\ge \tau^*$ and
$$
{1\over 4L}\le \sigma\le 1-{1\over 4L}
$$
then a minimizer is the characteristic function $u_{\min}=\chi_{[0,\sigma L]}$. The energy is 
$F_\tau^\#(u_{\min})=2$. Note that all other minimizers are the translations 
of $u_{\min}$;

C) otherwise, if $\sigma\le{1\over 2}$,
the minimum is obtained by minimizing on functions of the form $u=
{\sigma L\over y}\chi_{[0,y]}$. In this case, $F_\tau^\#(u_{\rm min})= 4\sqrt{\sigma L}$,
and again all other minimizers are obtained by translation;

D) finally, in the remaining cases, the minimum is obtained by minimizing on functions of the form $u=1-{(1-\sigma) L\over y}\chi_{[0,\sigma y]}$. In this case, $F_\tau^\#(u_{\rm min})= 
4\sqrt{(1-\sigma)L}$,
and again all other minimizers are obtained by translation.


\begin{remark}[size effect in the periodic case]\label{rem-minimi}\rm
We limit our description to $\sigma\le 1/2$,  
and we consider $\tau_\ast\in (0,\frac{1}{2})$, analyzing separately the limit cases 
$\tau_\ast=0$ and $\tau^\ast=\frac{1}{2}$.
\bigskip

{\bf Small-size domain:}  if $L<\tau_\ast$ we have

i) for $\sigma<{L}$ the minimizers 
are in the case {(}C) above 
\begin{equation} \label{sigma-piccolo}
\sqrt{\sigma L} \,\chi_{[s,\sqrt{\sigma L}+s]} \ \ \hbox { with }\ \ s\leq L-\sqrt{\sigma L}; 
\end{equation}

ii) for $\sigma =L$ the expressions in (\ref{sigma-piccolo}) define the constant function 
$u=L$, which is 
the only minimizer of the energy.
Note that the same holds, if $\sigma\geq\frac{1}{2}$, in the case $\sigma=1-L$; 

\smallskip 

iii) for ${L}<\sigma<\tau_\ast$ the minimizer
is the constant $u=\sigma$;

\smallskip

iv) for $\sigma\geq\tau_\ast$ the minimizers are all monotone non-increasing $u$ satisfying the integral constraint and the boundary conditions $u(0)-u(L)\leq\tau_*$.
%

\smallskip
 \begin{figure}[h!]
\centerline{\includegraphics [width=5.8in]{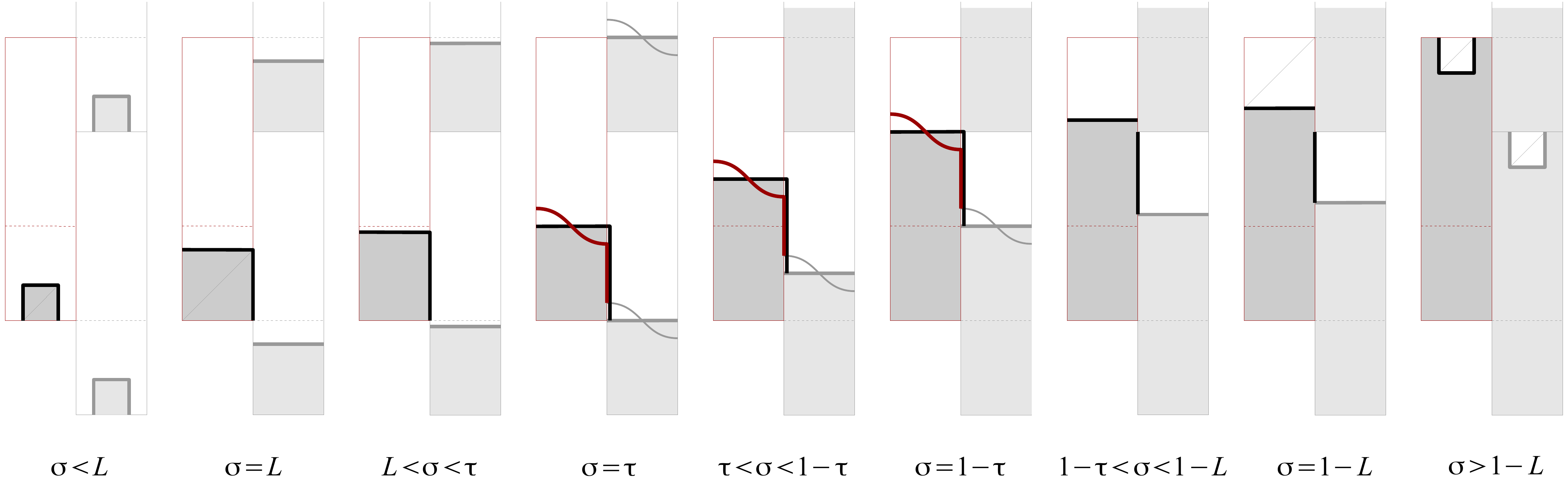}}
\caption{Minimizers of the energy with varying $\sigma$ for small-size domains}\label{piccoloP}\end{figure}
Note that at $\sigma={\tau_\ast}$, we have a discontinuity in the form of minimizers. The evolution of the form of the minimizers if $L<{\tau_\ast}$ is pictured in Fig.~\ref{piccoloP}. For the other cases the form of minimizers can be similarly described and we refer to the figures in Section \ref{sizeffect}
for the necessary changes.

\bigskip

{\bf Intermediate-size domain:} if $\tau_\ast<L<\tau^\ast$ 
then we have

i) for $\sigma< {(L+\tau_*)^2\over 4L}$ 
the minimizers are the functions in (\ref{sigma-piccolo});

\smallskip 

ii) for $\sigma = {(L+\tau_*)^2\over 4L}$ the minimizers are the functions of the form  
\begin{equation*}
\frac{L+\tau_\ast}{2} \,\chi_{[s,\frac{L+\tau_\ast}{2}+s]} 
\hbox { if }\ \ s\leq \frac{L-\tau_\ast}{2}
\end{equation*}
the constant $u=\sigma$ and all monotone non-increasing $u$ satisfying the integral constraint and the boundary conditions $u(0)-u(L)\leq\tau_*;$ 

\smallskip

iii) for ${(L+\tau_*)^2\over 4L} <\sigma$ the minimizers are all monotone non-increasing $u$ satisfying the integral constraint and the boundary conditions $u(0)-u(L)\leq\tau_*.$ 

\smallskip

The evolution of the form of the minimizers is similar to the situation described in Fig.~\ref{piccoloNP}  
for the non-periodic case, with a different critical threshold: here, the discontinuity appears 
at $\sigma={(L+\tau_*)^2\over 4L}$ and corresponds to a critical value ${L+\tau_\ast\over 2}$. 

\bigskip 

{\bf Large-size domain:} if $L>{\tau^\ast}$ then we have

i) for $\sigma<{1\over 4L}$ 
the minimizers are the functions in (\ref{sigma-piccolo});

\smallskip

ii) for $\sigma ={1\over 4L}$ the minimizers are the functions in (\ref{sigma-piccolo}) which in this case become  
\begin{equation*}
{1\over 2} \,\chi_{[s,{1\over 2}+s]} 
\ \ \hbox { if }\ \ s\leq L-{1\over 2} \qquad \hbox{ and } \qquad 
{1\over 2}\, \chi_{[0,{1\over 2}-s]\cup [{1\over 2},L]} 
\ \ \hbox { if }\ \ s>L-{1\over 2};
\end{equation*}
and the characteristic functions of the form $\chi_{[s,s+{1\over 4}]}$ and $\chi_{[0,{1\over 4}-s]\cup[s,L]}$; 

\smallskip

iii) for ${1\over 4L}<\sigma$ the minimizers are the characteristic functions of the form $\chi_{[s,s+{1\over 4}]}$ and $\chi_{[0,{1\over 4}-s]\cup[s,L]}$.

\smallskip

Again, 
the evolution of the form of the minimizers can be described as in Fig.~\ref{grandeNP}  
for the non-periodic case, with 
a discontinuity 
at $\sigma={1\over 4L}$.

\bigskip 

Note that if $\tau_\ast=0$ or $\tau_\ast=\frac{1}{2}$, 
then we only have two regimes (for domains with $L<1$ and with $L>1$) as in the non-periodic case.  
\end{remark}

\begin{remark}\rm 
Note that $0\in\tau(L)$ for any $L>0$ (see Remark \ref{rem-tau}), so that 
for all $\sigma$ 
$$\inf_{\tau\in\tau(L)} \min \{F^\#_\tau(u)\}=\min \{F^\#_0(u)\}$$ 
where the minimum is attained in $\tau=0$.
Moreover, the minimum value of the functional $F^\#_\tau$ is independent of $\tau$,
in the following cases: 
$$L\geq 1; \qquad L<1 \ \hbox{ and } \ L\leq \min\{\sigma,1-\sigma\}; 
\qquad 
L<1 \ \hbox{ and } \ \frac{L}{4} \geq \min\{\sigma, 1-\sigma\},$$
when it equals the minimum value of the functional $F^\#_0$.
\end{remark}

%
%
%
%
%
%
%
%

\end{document}